\documentclass[12pt]{article} \usepackage{fullpage}
\usepackage{amsmath,amssymb,amsthm} \usepackage[sort,numbers]{natbib}
\usepackage{mathrsfs} \usepackage{todonotes} \usepackage{authblk}
\usepackage[ruled,vlined]{algorithm2e} \usepackage{algpseudocode}
\usepackage{amssymb, colordvi} \usepackage{multirow} \usepackage{graphicx}
\usepackage{blkarray} \usepackage{tikz} \usetikzlibrary{decorations}
\usepackage{centernot}
\pgfdeclaredecoration{multiple arrows}{draw}{
  \state{draw}[width=\pgfdecoratedinputsegmentlength]{ \draw [multiple arrows
  path/.try] (0,0) -- (\pgfdecoratedinputsegmentlength,0); } }
  \tikzset{multiple arrows/.style={multiple arrows path/.style={#1},
decoration=multiple arrows, decorate}}

\numberwithin{equation}{section} \newtheorem{theorem}{Theorem}
\numberwithin{theorem}{section}

\newtheorem{proposition}[theorem]{Proposition}
\newtheorem{example}[theorem]{Example} \newtheorem{lemma}[theorem]{Lemma}
\newtheorem{corollary}[theorem]{Corollary}
\newtheorem{definition}[theorem]{Definition}

\newtheorem{problem}[theorem]{Problem} \newtheorem{question}[theorem]{Question}

\newcommand{\Z}{\mathbb{Z}} \newcommand{\calO}{\mathcal{O}}

\DeclareMathOperator{\GL}{GL} 
 \DeclareMathOperator{\minor}{minor}
 
\DeclareMathOperator{\sign}{sign}

 \newcommand{\R}{\mathbb{R}}
\newcommand{\N}{\mathbb{N}}

\DeclareMathOperator{\Vol}{Vol}

\newenvironment{myenumerate}{ \vspace{-3pt} \begin{enumerate}
  \setlength{\itemsep}{0pt} \setlength{\parskip}{0pt}
    \setlength{\parsep}{-8pt}} {\end{enumerate} \vspace{-3pt} }

\begin{document}

\title{
Sparse Polynomial Systems with many Positive Solutions from Bipartite
Simplicial Complexes
} \author[1]{Fr\'ed\'eric Bihan and Pierre-Jean Spaenlehauer}

\date{}

\maketitle

\begin{abstract}

Consider a regular
triangulation of the convex-hull $P$ of a set $\mathcal A$ of $n$ points in
$\R^d$, and a real matrix $C$ of size $d \times n$. A version of Viro's
method allows to construct from these data an unmixed polynomial system with
support $\mathcal A$ and coefficient matrix $C$ whose number of positive solutions is bounded from below by the number of $d$-simplices which are \emph{positively decorated} by $C$.
We show that all the $d$-simplices of a triangulation can be positively
decorated if and only if the triangulation is \emph{balanced}, which in turn is equivalent to
the fact that its dual graph is \emph{bipartite}.  This allows us to identify,
among classical families, monomial
supports which admit
\emph{maximally positive systems}, \emph{i.e.} systems all toric complex
solutions of which are real and positive. These families give some evidence in favor of a
conjecture due to Bihan.  We also use this technique in order to construct fewnomial 
systems with many positive solutions. This is done by considering a
simplicial complex with bipartite dual graph included in a regular
triangulation of the cyclic polytope.
\end{abstract}

\section{Introduction}

Real solutions of multivariate polynomial systems are central objects in many
areas of mathematics. Positive solutions (\emph{i.e.}
solutions all coordinates of which are real and positive) are of special interest as
they contain meaningful information in several applications, \emph{e.g.} 
robotics, optimization, algebraic statistics, etc.
In the 70s,
foundational results by Kouchnirenko~\cite{Kouch}, Khovanskii \cite{Khov} and
Bernshtein \cite{Bern} have laid theoretical ground for the study of
the algebraic structure of sparse polynomial equations in strong connection
with the development of toric and tropical
geometry.
These breakthroughs opened the door to computational techniques for sparse
elimination \cite{St4, HSt, CanEmi, St2, FSS}.

Let $\mathcal A\subset \Z^d$ be a finite point configuration.  We consider
unmixed sparse polynomial systems $f_1(X_1,\ldots, X_d)=\dots = f_d(X_1,\ldots,
X_d) = 0$ with support $\mathcal A$. This means that $\mathcal A$ coincides
with the set of exponent vectors $\mathbf a$ of the monomials $X^{\mathbf a}$
appearing in each equation.  Kouchnirenko's theorem \cite{Kouch} states that the
number of toric complex solutions (no coordinate is zero) which are
non-degenerate (the Jacobian matrix of the system is invertible at the
solution) is bounded by the normalized volume of the convex hull of~$\mathcal
A$.

Viro's method \cite{Vir} (see also \cite{R, St3,B} for instance) is one of the
roots of tropical geometry and has been used with great success for
constructing real algebraic varieties with interesting topological types.  It
allows to recover under certain conditions the topological type for $t$ close
to $0$ of a real algebraic variety defined by a system whose coefficients
depend polynomially on a positive parameter $t$.
Here we apply a version of Viro's method which has already been used in
\cite{St}.
Given any finite configuration $\mathcal A=\{{\mathbf a_1,\ldots, \mathbf
a_n}\} \subset \Z^d$, where $n=\lvert\mathcal A\rvert$, an unmixed real
polynomial system with support $\mathcal A$ can be written as $C \cdot
(X^{\mathbf a_1}, \ldots, X^{\mathbf a_n})^T=0$, where $C$ is a real matrix of
size $d\times n$ called \emph{coefficient matrix}. Given a regular
triangulation of the convex-hull of $\mathcal A$ associated with a height
function $h: {\mathcal A} \mapsto \R$, we look at the deformed system $C \cdot
(t^{h(\mathbf a_1)}x^{\mathbf a_1}, \ldots, t^{h(\mathbf a_n)}x^{\mathbf
a_n})^T=0$.
For $t>0$ sufficiently small, the number of complex (resp., real, positive)
toric solutions of this deformed system is at least the total number of complex
(resp., real, positive) toric solutions of the sub-systems obtained by
truncating the initial system to the $d$-simplices of the triangulation.  We
note that as far as we are only concerned with positive solutions, this
construction works in the same manner if we allow real exponent vectors
\emph{i.e.} if $\mathcal A \subset \R^d$.  If the triangulation is unimodular,
which means that all $d$-simplices have normalized volume one, then all these
sub-systems are linear up to monomial changes of coordinates.  Generically, a
real linear system has one complex toric solution, which is in fact a real
solution.  Since the number of $d$-simplices in any unimodular triangulation of
the convex-hull of $\mathcal A$ is equal to its normalized volume, this
construction produces polynomial systems whose all toric complex solutions are
real \cite[Corollary 2.4]{St}.

One goal of the present paper was to analyze under which conditions this
construction produces polynomial systems whose all toric complex solutions are
positive.  Such polynomial systems are called \emph{maximally positive} in
\cite{B}, where a conjecture about their supports has been proposed
\cite[Conjecture 0.6]{B}.

\smallskip

{\bf Main results.} Consider a regular full-dimensional pure simplicial complex supported on a point configuration  $\mathcal A=\{{\mathbf a_1,\ldots, \mathbf a_n}\} \subset\R^d$ and a coefficient matrix $C$ of size $d\times n $ which
encodes a map $\mathcal A\rightarrow \R^d$. We call a $d$-simplex of this simplicial complex
\emph{positively decorated} by $C$ if the kernel of a $d\times(d+1)$
submatrix of $C$ corresponding to this simplex contains vectors all coordinates
of which are positive. This condition can be read off from the signs of
maximal minors of $C$. Said otherwise, the simplices that are positively
decorated can be identified on the oriented matroid associated to $C$.
Moreover, our construction produces a polynomial system whose number of
positive solutions is at least
the number of $d$-simplices of a regular triangulation of the convex-hull of $\mathcal A$
which are positively decorated by $C$. Our first observation is that any simplicial complex
supported on $\mathcal A$ whose all $d$-simplices are positively decorated has a dual graph
which is bipartite. This leads us to investigate three types of full-dimensional pure simplicial complexes: \emph{balanced} simplicial complexes have the property that their set of
vertices 
is $(d+1)$-coloriable (two adjacent vertices have different colors)
\cite[Section III.4]{Sta3}; \emph{positively decorated} simplicial complexes are
characterized by the existence of a coefficient matrix which positively
decorates all their $d$-simplices; finally, \emph{bipartite}
simplicial complexes are characterized by the property that their dual graphs are bipartite. We shall
see that balancedness implies that the complex can be positively decorated, which in turn implies that the complex
is bipartite. For triangulations, these three properties are equivalent.
Consequently, if a point configuration $\mathcal A$ admits a regular, unimodular
and balanced triangulation, then our construction produces maximally positive
polynomial systems with support $\mathcal A$. 


In order to illustrate this result,
we check that some classical families of polytopes, namely order polytopes,
hypersimplices, cross polytopes and alcoved polytopes, admit regular unimodular balanced triangulations and
thus provide point configurations supporting maximally positive systems. 
As a by-product, we verify that these classes
of point configurations have a basis of affine relations with coefficients in
$\{-2,-1,1,2\}$: this gives evidence in favor of Bihan's conjecture
\cite[Conjecture 0.6]{B}.

Interesting computational problems arise from this analysis: if $\Gamma$ is a
full-dimensional pure simplicial complex in $\R^d$ whose dual graph is connected, deciding
if it is
balanced or if its dual graph is bipartite is computationally easy. However,
deciding if it is positively decorable seems to be a nontrivial problem, which
can be restated as deciding the existence of a realizable oriented matroid
verifying conditions given by the combinatorial structure of the complex.
We also show that the problem of decorating a simplicial complex can be recasted as a low-rank matrix completion
problem with positivity constraints.

\medskip

We apply our results to the problem of constructing fewnomial systems with many positive solutions comparatively to their number of monomials. Let $d$ be the number of variables (and equations) and $d+k+1$ the total number of monomials of a polynomial system.
As a particular case of more general bounds, Khovanskii \cite{Khov} obtained an upper bound on the number of (non-degenerate) positive solutions of such a system which depends only on $d$ and $k$. This bound was later improved by Bihan and Sottile \cite{BS} to some constant times $2^{\binom{k}2}d^k$.
When $k\gg d$, taking $d$ univariate polynomials with distinct variables provides a construction with many positive solutions, while for $d\gg k$  the record construction is due to Bihan, Rojas and Sottile \cite{BRS}.  We focus here in the case $k=d$, where the best construction so far
is to consider a system with $d$ quadratic univariate equations with distinct
variables, yielding $2^d$ positive solutions.
By considering a subsimplicial complex of a regular triangulation of the cyclic
polytope, we construct a pure simplicial complex of dimension $d$ on $2d+1$
vertices whose dual graph is bipartite. Its number of $d$-simplices grows as
$O((1+\sqrt 2)^d/\sqrt d)$. Consequently, if this simplicial complex is
positively decorable, then there exists a system with at least that many
positive solutions. For $d=1,3,5$, we compute explicitly such
decorations and we ask the question whether they exist for any $d$. 
In particular, for $d=5$ this yields a system with $11$ monomials and $38$ positive
solutions.

\medskip

{\bf Related works.} Configurations of points $\mathcal A$ that support
maximally positive systems (systems such that all toric complex solutions are
positive) have been characterized when $\mathcal A$ is the set of vertices of a
simplex (see \emph{e.g.} \cite{B2}) or when $\mathcal A$ is a circuit, see
\cite{B}. When $\mathcal A$ is any finite subset of $\Z$, it follows from
Descartes' rule of signs that $\mathcal A$ should coincide with the
intersection of its convex-hull with $\Z$.  Based on these characterizations,
Bihan conjectured that if $\mathcal A\subset\Z^d$ is the support of a maximally
positive polynomial system, then there is a basis of affine relations for
$\mathcal A$ with coefficients at most $2$ in absolute value~\cite{B}. By
\cite[Lemma 7.6]{MilStu}, this is equivalent to saying that the homogeneous
toric ideal $I_{\mathcal A}$ associated to $\mathcal A$ can be written
$I_{\mathcal A} = J:\langle X_1\cdots X_n\rangle^\infty$, where $J$ is
generated by binomials with exponents at most $2$.  Balanced regular
triangulations have also been used in \cite{SS} in order to get lower bounds
for the number of real solutions of Wronski polynomial systems.  Namely, given
a regular balanced triangulation of a convex polytope in $\R^d$, a Wronski
polynomial system is an unmixed polynomial system where all polynomials have
the form $\sum_{i=0}^d c_i \,\varphi_i(X)$ with $c_0,\ldots,c_d \in \R$ and
$\varphi_i(X)$ is a linear combination with positive coefficients of monomials
with given color $i$. Since the triangulation is bipartite, all its
$d$-simplices get a sign $\pm$ so that two adjacent $d$-simplices have opposite
signs. Soprunova and Sottile showed in \cite{SS} that under certain conditions
on the polytope the absolute value of the difference between the number of
positive and negative odd normalized volume $d$-simplices of the triangulation
provides a lower bound on the number of real solutions of any Wronski
polynomial system associated to this triangulation. In fact, Sottile informed
us that the existence of maximally positive Wronski polynomial systems when the
support admits a regular balanced unimodular triangulation follows from
\cite[Lemma 3.9]{SS}. Notice that our construction can be used to produce maximally positive
systems which are not Wronski polynomial systems. In general, triangulations
need not be balanced, but under some conditions, they admit a minimal branched
balanced covering. This has been investigated by Izmestiev and Joswig
\cite{Izmes, IzJ} for combinatorial $3$-manifolds.  The connection between the
oriented matroid defined by the matrix of coefficients and the number of
positive solutions has been investigated by M\"uller et al. in \cite[Theorem
1.5]{MulDic}, where they give a sufficient condition on this oriented matroid
for a sparse system to have at most one positive solution.

{\bf Organization of the paper.} Section \ref{sec:oriented} focuses on
simplices and describes the construction of Viro's system and its relation with
the oriented matroid associated to a given coefficient matrix $C$. In Section
\ref{sec:dualgraph}, the relationship between balanced, positively decorated
and bipartite simplicial complexes is investigated.  Section
\ref{sec:maxpos} focuses on unimodular and regular triangulations of
classical polytopes, and we identify classes of polytopes for which there
exists such a triangulation which is balanced, yielding construction of 
maximally positive polynomial systems.  In Section
\ref{sec:cyclic}, we focus on the cyclic polytope and we propose a
construction of a bipartite subsimplicial complex with many simplices. We finish in Section \ref{sec:comput} by
relating the problem of positive decorability of a simplicial complex with two
computational problems: the existence of realizable oriented matroids
satisfying a specific condition, and the low-rank matrix completion problem
with extra positivity constraints.

\paragraph{Acknowledgements.} We are grateful to Alin Bostan and Louis
Dumont for their help with the computation of diagonals of bivariate series 
and asymptotic estimates of the growth of their coefficients. We also thank
\'Eric Schost, Frank Sottile and Bernd Sturmfels for helpful discussions.

\section{Positively decorated simplices}\label{sec:oriented}

We start by focusing on systems of $d$ equations in $d$ variables involving
$d+1$ monomials. This case corresponds to simplices and they shall serve as building
blocks which will be glued to form simplicial complexes. Such systems are equivalent
to linear systems up to a monomial map and
the positivity of their solution can be read off from the
signs of the maximal minors of the matrix recording the coefficients of the
system.

\begin{definition} A $d\times (d+1)$ matrix $M$ with real entries is called
  \emph{oriented} if all the values $(-1)^i{\minor}(M,i)$ are nonzero and have 
  the same sign, where $\minor(M,i)$ is the determinant of the square matrix
  obtained by removing the $i$-th column.  \end{definition}

The terminology ``oriented'' follows from the fact that for $d=2$, the signs of
the minors determine orientations of the edges of a $2$-dimensional simplex
compatible with an orientation of the plane (see Figure \ref{fig:color}).
The
following proposition is elementary, but it plays a central role in the sequel
of the paper.

\begin{proposition} \label{P:invariant} Let $M$ be a full rank $d\times (d+1)$
matrix with real entries. Then the following statements are equivalent:
\begin{myenumerate} \item the matrix $M$ is oriented; 
  \item for any
      $A\in\GL_n(\R)$, $A\cdot M$ is an oriented matrix; \item for any
      permutation matrix $P\in\mathfrak S_{d+1}$, $M\cdot P$ is an oriented
    matrix; \item all the coordinates of any non-zero vector in the kernel of
      the matrix are non-zero and share the same sign;  \item there exists $i
      \in \{1,\ldots,d+1\}$ such that the $i$-th column vector of $M$ belongs
    to the interior of the negative cone generated by the other column vectors
  of $M$;  \item for any $i \in \{1,\ldots,d+1\}$, the $i$-th column vector of
    $M$ belongs to the interior of the negative cone generated by the other
    column vectors of $M$.  \end{myenumerate} \end{proposition}

\begin{proof} The equivalence $(1)\Leftrightarrow (4)$ follows from Cramer's
  rule. $(2)\Rightarrow (1)$ and $(3)\Rightarrow (1)$ are proved directly by
  instanciating $A$ and $P$ to the identity matrix.  $(1)\Rightarrow (2)$
  follows from $$\sign((-1)^i{\minor}(A\cdot M,i)) = \sign(\det(A))\cdot
  \sign((-1)^i{\minor}(M,i)).$$ $(3) \Leftrightarrow (4)$ is a consequence of
  the fact that permuting the columns of $M$ is equivalent to permuting the
  coordinates of the kernel vectors. Finally, the equivalence between $(4)$,
  $(5)$ and $(6)$ is obvious once we have noticed that a positive vector $(x_1,\ldots,x_{d+1})$
  belongs to the kernel of $M$ if and only if $\mathbf m_i=\sum_{j \neq i}
  -\frac{x_j}{x_i}\mathbf m_j$ assuming $x_i \neq 0$, where $\mathbf
  m_1,\ldots, \mathbf m_{d+1}$ are
  the column vectors of $M$.  \end{proof}

We let $X$ denote the set of variables $\{X_1,\ldots, X_d\}$. A
solution $\mathbf v=(v_1,\ldots, v_d)$ of a system $f_1(X) = \dots = f_d(X) =
0$ is called \emph{non-degenerate} if all the functions $f_i$ are~$C^1$ at
$\mathbf v$ and the Jacobian matrix of $(f_1,\ldots, f_d)$ is invertible at
$\mathbf v$. Throughout the paper, $\mathcal A=\{\mathbf a_1,\ldots, \mathbf
  a_{n}\}\subset \Z^d$ denotes a
finite point configuration, and the coordinates of these points are
recorded in a $d\times n$ matrix $A$ (we assume
that an ordering of the points has been arbitrarily fixed). We let
$\widetilde A$ denote the matrix obtained by adding a first row whose
entries are all $1$. For $\mathbf a=(a_1,\ldots, a_d)\in\Z^d$, the
shorthand $X^{\mathbf a}$ stands for the monomial $X_1^{a_1}\dots
X_d^{a_d}$. 

\begin{proposition}\label{prop:posorthantsimplex}
Let $\mathcal A=\{\mathbf a_1,\ldots, \mathbf a_{d+1}\}\subset\Z^d$ and  
$$f_i(X)=\sum_{j=1}^{d+1}C_{i,j} X^{\mathbf a_j},\quad 1\leq i\leq d$$
be a system of $d$ Laurent polynomials with real coefficients involving $d+1$ monomials.
If $\widetilde A$ is invertible,
then the system $f_1(X)=\dots=f_d(X)=0$
 has one non-degenerate solution in the positive orthant if and only if the $d\times (d+1)$ matrix $C$ recording the coefficients of the system is oriented. 
\end{proposition}

\begin{proof}
To any invertible $(d+1)\times (d+1)$ real matrix $S$ with columns $(\mathbf
s_1,\ldots,\mathbf s_{d+1})$, 
we associate the bijection of the positive orthant
$$\begin{array}{rccc}
  \mu_S:&\R_+^{d+1}&\rightarrow&\R_+^{d+1}\\
  & \mathbf x&\mapsto&(\mathbf x^{\mathbf s_1},\ldots, \mathbf x^{\mathbf s_{d+1}}).
\end{array}$$
Its inverse map is $\mu_{S^{-1}}$. Let $\ell_1(X_0,\ldots,
X_{d})=\dots=\ell_d(X_0,\ldots, X_{d})=0$ be the linear system defined by
$$\ell_i(X_0,\ldots, X_{d})=\sum_{j=0}^{d}C_{i,j+1} X_{j}.$$
By Cramer's rule, this system has a solution in the positive orthant if and only if the 
matrix $C$ is oriented. Since $(\ell_i\circ\mu_{\widetilde
A})(1, X_1,\ldots,
X_d)=f_i(X)$, the positive solutions 
of $f_1(X)=\dots=f_d(X)=0$ are in bijection with those of $\ell_1(1,X_1,\ldots,
X_d)= \dots = \ell_d(1, X_1,\ldots, X_d)=0$.
\end{proof}

Let $\mathcal A=\{\mathbf a_1,\ldots,\mathbf a_n\} \subset\R^d$ be a finite point configuration, and assume that the convex-hull of $\mathcal A$ is a full-dimensional polytope
$Q$. Let $(\Gamma,
\nu)$ be a regular triangulation of the convex hull of $\mathcal A$
\emph{i.e.}  $\Gamma$ is a triangulation and $\nu$ is a convex
function, linear on each simplex of $\Gamma$, but not linear on the union of two different maximal simplices of $\Gamma$.
Regular triangulations are sometimes called coherent or convex in the literature. 
Let $C$ be a $d\times n$ matrix with real entries. We say that $C$ \emph{positively decorates} a $d$-simplex $\Delta={\rm
conv}(\mathbf a_{i_1},\ldots,\mathbf a_{i_{d+1}})\in\Gamma$ if the $d\times(d+1)$ 
submatrix of $C$ given by its columns $\{i_1,\ldots, i_{d+1}\}$ is oriented.

Consider the following family of polynomial systems parametrized by a
positive real number $t$:
\begin{equation}
\label{E:Virosystem}
f_{1,t}(X)=\cdots=f_{d,t}(X)=0,
\end{equation}
where
$$f_{i,t}(X)=\sum_{j=1}^n C_{ij}t^{\nu(\mathbf a_j)}X^{\mathbf a_j} \in
\R[X_1,\ldots,X_d] , \quad i=1,\ldots,d,\quad t>0.$$
For each positive real value of $t$, this system has support included in $\mathcal A$.

\begin{theorem}\label{thm:nbpossols}
There exists $t_0\in\R_+$ such that for all $0<t<t_0$ the number
of non-degenerate solutions of \eqref{E:Virosystem} contained in the positive orthant
is bounded from below by the number of maximal simplices in $\Gamma$ which are
positively decorated by $C$.
\end{theorem}
\begin{proof}
Let $\Gamma_1,\ldots,\Gamma_k$ be the maximal simplices in $\Gamma$ which are
positively decorated by $C$. Let $\nu_{\ell}$ be the restriction of $\nu$ to $\Gamma_{\ell}$ for
$\ell=1,\ldots,k$. The function $\nu_{\ell}$ is affine hence there exist $\mathbf a_{\ell}=(a_{1 \ell},\ldots,a_{d \ell}) \in \R^d$ and $b_{\ell} \in \R$ such that
$\nu_{\ell}(x)=\langle a_{\ell},x \rangle +b_{\ell}$ for any $x=(x_1,\ldots,x_d) \in \Gamma_{\ell}$. 
Moreover, since $\nu$ is convex and not affine on the union of two distinct maximal simplices of $\Gamma$, 
setting $Y=(Y_1,\ldots,Y_d)$ and $Yt^{-\mathbf a_{\ell}}=(Y_1t^{-a_{1 \ell}},\ldots,
Y_dt^{-a_{d \ell}})$ we get
$$
\frac{f_{i,t}(Yt^{-\mathbf a_{\ell}})}{t^{b_{\ell}}}
=
f_{i}^{(\ell)}(Y)+r_{i,t}(Y), \quad i=1,\ldots,d,$$

where $f_{i}^{(\ell)}(Y)=\sum_{w_j \in \Gamma_{\ell}} C_{ij}Y^{w_j}$ and
$r_{i,t}(Y)$ is a polynomial in $\R[Y_1,\ldots,Y_d]$ whose coefficients are
products of real numbers by positive powers of $t$. Since $\Gamma_{\ell}$ is
positively decorated by $C$, by Proposition \ref{prop:posorthantsimplex} the system
$f_{1}^{(\ell)}(Y)=\cdots=f_{d}^{(\ell)}(Y)=0$ has one non-degenerate positive solution. Let $K$ be a compact set in the positive orthant which
contains all the non-degenerate positive solutions of the systems
$f_{1}^{(\ell)}(Y)=\cdots=f_{d}^{(\ell)}(Y)=0$ for $\ell=1,\ldots,k$. If $t>0$
is small enough, the sets $t^{-\mathbf a_{\ell}}\cdot K= \{(y_1t^{-a_{1
\ell}},\ldots, y_dt^{-a_{n \ell}}) \, | \, (y_1,\ldots,y_d) \in K\}$,
$\ell=1,\ldots,k$, are pairwise disjoint and each one contains at least one
non-degenerate positive solution of the system \eqref{E:Virosystem}.
\end{proof}

Recall that $Q$ is the convex-hull of $\mathcal A$ and that $\Vol(\,
\cdot \, )$ stands for the Euclidean volume of $\R^d$ multiplied by
$d!$.  The triangulation $\Gamma$ is called \emph{unimodular} if $A$
has integer entries and any maximal simplex $\Gamma_i \in \Gamma$
verifies $\Vol(\Gamma_i)=1$.  Sturmfels \cite{St} showed that if $A$
has integer entries and $\Gamma$ is unimodular then for $t>0$ small
enough the system \eqref{E:Virosystem} has exactly $\Vol(Q)$
non-degenerate solutions with non-zero real coordinates, and no other
solution with non-zero complex coordinates.
The following proposition shows that if moreover the triangulation is
positively decorated, then all these solutions are positive. We stress that the
existence of maximally positive systems for monomial supports admitting a
regular, unimodular and regular triangulation was already proved in \cite[Lemma
3.9]{SS} using Wronski systems.

\begin{proposition}\label{P:maximally}
If $A$ is a matrix with integer entries, $\Gamma$ is a unimodular regular
triangulation and all its $d$-simplices are positively decorated by $C$, then for $t>0$ small enough
the system \eqref{E:Virosystem} has exactly $\Vol(Q)$ non-degenerate positive solutions, and no other solution with non-zero complex coordinates.
\end{proposition}
\begin{proof}
By Theorem \ref{thm:nbpossols}, the system \eqref{E:Virosystem} has at least $\Vol(Q)$ non-degenerate solutions in the positive orthant
for $t>0$ small enough. On the other hand, the system \eqref{E:Virosystem} has at most $\Vol(Q)$ non-degenerate solutions with non-zero complex coordinates by Kouchnirenko Theorem
\cite{Kouch}.
\end{proof}

Polynomial systems whose all non-degenerate solutions with non-zero
complex coordinates are contained in the positive orthant are called
\emph{maximally positive} in \cite{B}. We shall put a special focus on
classical polytopes admitting maximally positive polynomial systems
in Section \ref{sec:maxpos}.

\section{Bipartite dual graphs and balanced triangulations}\label{sec:dualgraph}

The aim of this section is to show how Theorem
\ref{thm:nbpossols} may be used to construct systems of polynomials
with prescribed support and many positive real solutions. 
Throughout this paper, by simplicial complex, we always
mean a full-dimensionsal pure geometric simplicial complex embedded in $\R^d$, see \cite[Definition 2.3.5]{Maund}.

\begin{definition}
Let $\mathcal A\subset \R^d$ be a finite configuration of points,
represented by a $d\times n$ matrix~$A$.  A
\emph{positively decorated simplicial complex} supported on $\mathcal A$ is a pair
$(\Gamma, C)$, where $\Gamma$ is a simplicial complex whose
vertex set is a subset of $\mathcal A$ and $C$ is a
$d\times n$ matrix such that every submatrix of
size $d\times (d+1)$ corresponding to a $d$-simplex in
$\Gamma$ is an oriented matrix.
\end{definition}

Throughout this paper, we represent a pure $d$-dimensional simplicial complex as a finite set
$\{\tau_1,\ldots, \tau_k\}$, where $\tau_i\subset \{1,\ldots,
n\}$ has cardinality $d+1$. 
For $\tau\in\Gamma$ a $d$-simplex and $C$ a coefficient matrix
associated to the point configuration $\mathcal A$, we let $C_\tau$
denote the $d\times (d+1)$ submatrix of $C$ whose columns correspond to
the $d+1$ vertices in $\tau$.
Let $\tau$ be a $d$-simplex in $\R^d$ with vertices $\mathbf
a_1,\ldots,\mathbf a_{d+1}$, and
let $A$ be the corresponding $d \times (d+1)$ matrix. Let $\mathbf a$ be any point in
the interior of $\tau$ and let $D_{\tau}$ be the $d \times (d+1)$ matrix with
columns $\mathbf a_1-\mathbf a,\mathbf a_2-\mathbf a,\ldots,\mathbf
a_{d+1}-\mathbf a$. 
Clearly, the oriented matroid defined by $D_{\tau}$ does not depend on the
choice of $\mathbf a$.  We say that two matrices of the same size define the same
(resp., opposite) oriented matroid if they have the same bases and two bases
corresponding to the same set of columns have determinant of the same (resp.,
opposite) sign.

\begin{lemma}\label{L:orientedmatroid} The matrix $D_{\tau}$ is oriented.
  Therefore, a $d \times (d+1)$ matrix $C_\tau$ positively decorates $\tau$ if and only if
  either $C_\tau$ and $D_{\tau}$ define the same oriented matroid, or $C_\tau$
  and $D_{\tau}$ define opposite oriented matroids.  \end{lemma}

\begin{proof}
Let $A(i)$ be the matrix obtained by removing the $i$-th column
from the matrix with columns (in this order) $\mathbf a_1-\mathbf a_i, \mathbf
a_2-\mathbf a_i,\ldots,\mathbf a_{d+1}-\mathbf a_i$.
Clearly, $\mbox{minor}(D_{\tau},i) \cdot \mbox{det} \, A(i) >0$. On the other hand, we compute
that $(-1)^{i+1} \mbox{det} \, A(i)$ coincide for $i=1,\ldots,d+1$ with the determinant
of the $(d+1) \times (d+1)$ matrix $\widetilde{A}$ obtained by adding to $A$ a
first row of ones.
Thus, all $(-1)^i \cdot \mbox{minor}(D_{\tau},i)$ have the same sign, which
means that $D_{\tau}$ is oriented. It follows that $C_\tau$ positively
decorates $\tau$ if and only if either $\mbox{minor}(C_\tau,i) \cdot \mbox{minor}(D_{\tau},i) >0$ for all $i$, or $\mbox{minor}(C_\tau,i) \cdot \mbox{minor}(D_{\tau},i) <0$ for all $i$.
\end{proof}

\begin{lemma} \label{L:bipartite} Let $\tau$ and $\tau'$ be two $n$-simplices
  in $\R^n$ with a common facet.  Assume that the simplicial complex
  $\{\tau,\tau'\}$ is positively decorated by a matrix $C$.  If $C_\tau$ and $D_{\tau}$
  define the same oriented matroid, then $C_{\tau'}$ and $D_{\tau'}$ define the
  opposite oriented matroids.  \end{lemma} 

\begin{proof} Due to Proposition \ref{P:invariant}, we may permute
  simultaneously the columns of $A$ and $C$ so that $\tau$ corresponds to the
  columns $1,\ldots,d,d+1$ and $\tau'$ corresponds to the columns
  $1,\ldots,d,d+2$. Thus, the submatrix $C_{\tau}$ given by the columns
  $1,\ldots,d,d+1$ of $C$ and the submatrix $C_{\tau'}$ given by the columns
  $1,\ldots,d,d+2$ of $C$  are oriented.

Choose two points $\mathbf a$ and $\mathbf a'$ in the interior of $\tau$ and $\tau'$
respectively  which are symmetric with respect to the common facet with
vertices $\mathbf a_1,\ldots,\mathbf a_d$.  Then, $\mbox{det}(\mathbf
a_1-\mathbf a,\mathbf a_2-\mathbf a,\ldots,\mathbf a_d-\mathbf a)$ and
$\mbox{det}(\mathbf a_1-\mathbf a',\mathbf a_2-\mathbf a',\ldots,\mathbf
a_d-\mathbf a')$ have opposite signs (a hyperplane
symmetry has negative determinant). This means that
$\mbox{minor}(D_{\tau},d+1)$ and $\mbox{minor}(D_{\tau'},d+1)$ have opposite
signs. On the other hand, $\mbox{minor}(C_{\tau},d+1)$ and
$\mbox{minor}(C_{\tau'},d+1)$ are equal. It follows that if $C_{\tau}$ and
$D_{\tau}$ define the same oriented matroid, then $C_{\tau'}$ and $D_{\tau'}$
define the opposite oriented matroids.\end{proof}

The {\it dual graph} of a pure simplicial complex is the adjacency graph of its
simplices of maximal dimension.

\begin{proposition}\label{prop:simplicialcomploriented}
Let $\Gamma$ be a pure simplicial complex of dimension $d$ in $\R^d$. If there exists a
matrix $C$ such that the pair $(\Gamma, C)$ is a positively decorated
simplicial complex, then the dual graph of $\Gamma$ is bipartite.
\end{proposition}

\begin{proof} This is a straightforward consequence of Lemma \ref{L:bipartite}.
  In fact, the proofs of Lemma \ref{L:bipartite} and Lemma
  \ref{L:orientedmatroid} show how one can associate a sign $s(\tau)$ to each
  $d$-simplex $\tau$ of $\Gamma$ so that any two adjacent (with a common
  facet) such simplices have opposite signs.  Let $\widetilde{A}$ denote the matrix
  obtained by adding a first row of $1$ to $A$, and let $\widetilde{A}_{\tau}$ be
  the square submatrix corresponding to a $d$-simplex $\tau$ of $\Gamma$. Let
  $C_{\tau}$ be the submatrix of $C$ corresponding to $\tau$. Since $(\Gamma,
  C)$ is positively decorated, all $(-1)^{i+1} \mbox{minor}(C_{\tau},i)$ have the same sign
  $s(C_{\tau}) \in \{\pm 1\}$. We define then $s(\tau)$ as the sign of the
  product $\mbox{det} (\widetilde{A}_{\tau}) \cdot s(C_{\tau})$. If we denote by
  $\widetilde{C}_{\tau}$ the matrix obtained by adding a first row of $1$ to
  $C_{\tau}$, we immediately obtain that $s(\tau)$ is the sign of $\mbox{det}
  (\widetilde{A}_{\tau}) \cdot \mbox{det} (\widetilde{C}_{\tau})$.  \end{proof}

A large class of simplicial complexes with
bipartite dual graphs is provided by \emph{balanced} simplicial complexes.

\begin{definition} \cite[Section III.4]{Sta3}A $(d+1)$-coloration of a simplicial complex $\Gamma$
  supported on $\mathcal A$ is a map $\gamma:\mathcal A\rightarrow \{1,\ldots,
  d+1\}$ (or more generally a map from $\mathcal A$ to a set with $d+1$
  elements) such that $\gamma(\mathbf a_1)\neq \gamma(\mathbf a_2)$ for any
  edge $(a_1,a_2)$ of $\gamma$. 
A $d$-dimensional simplicial complex $\Gamma$ supported on $\mathcal A$ is
called {\it balanced} if there exists such a coloration.  
\end{definition}

Such colorations are sometimes called
  \emph{foldings} since they can extended to a map from $\Gamma$ to the $d$-dimensional
  standard simplex which is linear on each $d$-simplex of $\Gamma$.
Similarly, balanced triangulations are sometimes referred to as {\it
foldable} triangulations, see \emph{e.g.} \cite{JZ} and references within.

\begin{proposition}\label{prop:orientcoloring} Let $\mathbf e_i$ be the $i$-th
  canonical basis vector of $\R^d$ and $\mathbf e_{d+1}$ be the vector
  $(-1,\ldots, -1)$.  Let $\Gamma$ be a simplicial complex supported on
  $\mathcal A=\{\mathbf a_i\}_{1\leq i\leq \lvert\mathcal A\rvert}$. If $\Gamma$ is
  balanced and $\gamma:\mathcal A\rightarrow \{1,\ldots, d+1\}$ is a
  $(d+1)$-coloring of $\Gamma$, then the matrix $C$ with columns $(\mathbf
  e_{\gamma(\mathbf a_i)})_{1\leq i\leq \lvert\mathcal A\rvert}$ positively decorates $\Gamma$.
\end{proposition}

\begin{proof} By construction, every $d\times(d+1)$ submatrix of $C$
  corresponding to a $d$-simplex of $\Gamma$ is a column permutation of the
  $d\times(d+1)$ matrix with columns $(\mathbf e_1,\ldots, \mathbf e_{d+1})$.
  This latter matrix is oriented and hence the statement follows from
  Proposition \ref{P:invariant}.  \end{proof}

A pure simplicial complex $\Gamma$ is called {\it locally strongly connected} if 
the dual graph of the star of any vertex is connected.
By \cite[Proposition 6]{J} and \cite[Corollary 11]{J}, a locally
strongly connected and simply-connected complex $\Gamma$ on a finite set $\mathcal
A$ is balanced if and only if its dual graph is bipartite, see also
\cite[Theorem 5]{Izmes}. It is worth
noting that any triangulation of $\mathcal A$ (\emph{i.e.} a
triangulation of the convex-hull of $\mathcal A$ with vertices in
$\mathcal A$) is locally strongly connected and simply
connected. 

\begin{theorem} \label{T:manypositive} Assume that a finite full-dimensional
  point configuration $\mathcal A$ in $\R^d$  admits a regular triangulation
  $\Gamma$ which is balanced, or equivalently, whose dual graph is bipartite.
  Then there exists a polynomial system with support $\mathcal A$ whose number
  of positive solutions is at least the number of $d$-simplices of
  $\Gamma$. If furthermore $\mathcal A\subset\Z^d$ and the triangulation
  $\Gamma$ is unimodular, then there exists a polynomial system with support
  $\mathcal A$ which is maximally positive.  \end{theorem}

\begin{proof}
  This is a direct consequence of Theorem \ref{thm:nbpossols}, 
  Proposition \ref{prop:orientcoloring} and Kouchnirenko's theorem
  \cite{Kouch}.
\end{proof}
 
We finish this section by an explicit example of a polynomial system with
prescribed number of positive solutions obtained using this construction.

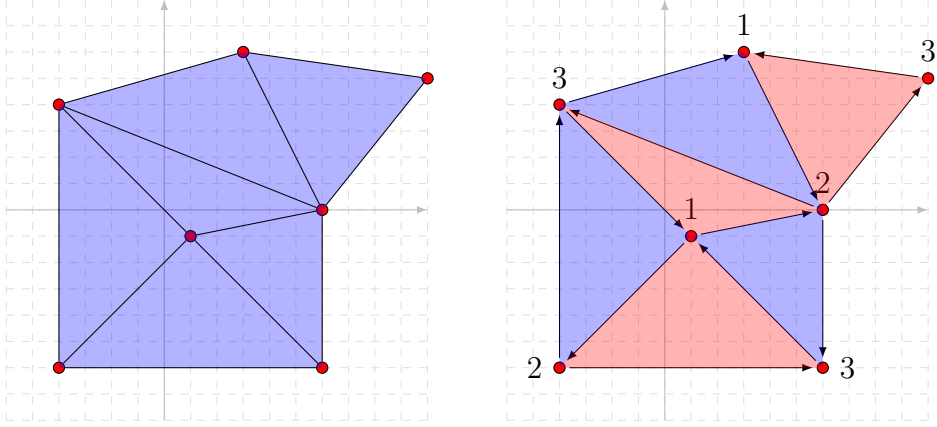
\begin{figure}
\centering
\begin{tikzpicture}[mynode/.style={circle, inner sep=1.5pt, fill=red}, scale = 0.7]
    \coordinate (Origin)   at (0,0);
    \coordinate (XAxisMin) at (-3,0);
    \coordinate (XAxisMax) at (5,0);
    \coordinate (YAxisMin) at (0,-4);
    \coordinate (YAxisMax) at (0,4);
    \draw [thin, lightgray,-latex] (XAxisMin) -- (XAxisMax);
    \draw [thin, lightgray,-latex] (YAxisMin) -- (YAxisMax);
  \draw[style=help lines, step= 0.5,dashed,opacity=0.5, lightgray] (-3,-4) grid (5,4);
  \node[draw,mynode] (A1) at (0.5,-0.5) {};
  \node[draw,mynode] (A2) at (-2,-3) {};
  \node[draw,mynode] (A3) at (-2,2) {};
  \node[draw,mynode] (A4) at (3,0) {};
  \node[draw,mynode] (A5) at (1.5,3) {};
  \node[draw,mynode] (A6) at (5,2.5) {};
  \node[draw,mynode] (A7) at (3,-3) {};
  \fill[opacity=0.3, blue] (A1.center) -- (A2.center) -- (A3.center) -- cycle;
  \fill[opacity=0.3, blue] (A1.center) -- (A4.center) -- (A3.center) -- cycle;
  \fill[opacity=0.3, blue] (A1.center) -- (A2.center) -- (A7.center) -- cycle;
  \fill[opacity=0.3, blue] (A1.center) -- (A4.center) -- (A7.center) -- cycle;
  \fill[opacity=0.3, blue] (A4.center) -- (A5.center) -- (A3.center) -- cycle;
  \fill[opacity=0.3, blue] (A4.center) -- (A5.center) -- (A6.center) -- cycle;
  \draw (A1) -- (A2) -- (A3) -- (A1) -- (A4) --(A3) -- (A5) -- (A4) -- (A6) -- (A5);

  \draw (A4) -- (A7) -- (A1);
  \draw (A2) -- (A7);
\end{tikzpicture}
\quad\quad
\begin{tikzpicture}[mynode/.style={circle, inner sep=1.5pt,
  fill=red}, scale = 0.7]
    \coordinate (Origin)   at (0,0);
    \coordinate (XAxisMin) at (-3,0);
    \coordinate (XAxisMax) at (5,0);
    \coordinate (YAxisMin) at (0,-4);
    \coordinate (YAxisMax) at (0,4);
    \draw [thin, lightgray, -latex] (XAxisMin) -- (XAxisMax);
    \draw [thin, lightgray,-latex] (YAxisMin) -- (YAxisMax);
  \draw[style=help lines, opacity=0.5, step= 0.5,dashed, lightgray] (-3,-4) grid (5,4);
  \node[draw,mynode, label=above:{$1$}] (A1) at (0.5,-0.5) {};
  \node[draw,mynode, label=left:{$2$}] (A2) at (-2,-3) {};
  \node[draw,mynode, label=above:{$3$}] (A3) at (-2,2) {};
  \node[draw,mynode, label=above:{$2$}] (A4) at (3,0) {};
  \node[draw,mynode, label=above:{$1$}] (A5) at (1.5,3) {};
  \node[draw,mynode, label=above:{$3$}] (A6) at (5,2.5) {};
  \node[draw,mynode, label=right:{$3$}] (A7) at (3,-3) {};
  \path[multiple arrows={shorten >=0.5mm, shorten <=0.5mm, -latex}] 
  (A1) -- (A2) -- (A3) -- (A1) -- (A4) --(A3) -- (A5) -- (A4) --
  (A6) -- (A5);
  \path[multiple arrows={shorten >=0.5mm, shorten <=0.5mm, -latex}] 
  (A4) -- (A7) -- (A1);

  \path[multiple arrows={shorten >=0.5mm, shorten <=0.5mm, -latex}] 
  (A2) -- (A7);
  \fill[opacity=0.3, blue] (A1.center) -- (A2.center) -- (A3.center) -- cycle;
  \fill[opacity=0.3, red] (A1.center) -- (A4.center) -- (A3.center) -- cycle;
  \fill[opacity=0.3, red] (A1.center) -- (A2.center) -- (A7.center) -- cycle;
  \fill[opacity=0.3, blue] (A1.center) -- (A4.center) -- (A7.center) -- cycle;
  \fill[opacity=0.3, blue] (A4.center) -- (A5.center) -- (A3.center) -- cycle;
  \fill[opacity=0.3, red] (A4.center) -- (A5.center) -- (A6.center) -- cycle;
\end{tikzpicture}
\caption{The balanced simplicial complex from Example \ref{ex:simcomp6}, a $3$-coloration of its
  $1$-skeleton, and the induced orientations of its edges.\label{fig:color}}
\end{figure}

\begin{example}\label{ex:simcomp6}
Let $d=2$, $\mathcal A=(X Y^{-1}, X^{-4} Y^{-6}, X^{-4} Y^4, X^6, X^{3}Y^6, X^{10} Y^{5}, X^6Y^{-6})$, $\Gamma=\{ \{1,2,3\}, \{1,3,4\},\{3, 4, 5\}, \{4,5,6\}, \{1, 2,7\}, \{1, 4,7\}\}$. Choosing heights 

$$\begin{array}{rcl}
  \nu(1, -1) &=& 0\\
  \nu(-4, -6) &=& 0
\end{array}\quad
\begin{array}{rcl}
  \nu(-4, 4) &=& 0\\
  \nu(6,0) &=& 3\\
  \nu(3,6) &=& 5
\end{array}\quad
\begin{array}{rcl}
  \nu(10,5) &=& 10\\
  \nu(6,-6) &=& 2
\end{array}$$

provides us with a regular triangulation of $\mathcal A$ which has the balanced
subsimplicial complex described in Figure \ref{fig:color}.
Applying the construction of Theorem \ref{thm:nbpossols} and Proposition \ref{prop:orientcoloring}, we obtain the following system, depending on a parameter $t$:

$$\begin{array}{rcl}
f_1 &=& XY^{-1}- X^{-4} Y^4 + t^5X^{3}Y^6 - t^{10}X^{10} Y^{5} - t^2 X^6Y^{-6}\\
f_2 &=& X^{-4} Y^{-6} - X^{-4} Y^4 + t^3X^6 - t^{10}X^{10} Y^{5} - t^2 X^6Y^{-6}
\end{array}$$
which has at least six solutions in the positive orthant for $t$ sufficiently
small. Setting $t=1/1000$ and using Gr\"obner bases library {\tt FGb} \cite{FGB} and
the real solver {\tt rs\_isolate\_sys} \cite{RS} in Maple confirms that this system has indeed six positive solutions. 
\end{example}

At this point, we would like to recapitulate the relationship between the
properties of simplicial complexes studied in this section. As discussed above,
balanced simplicial complexes are always positively decorable (Proposition \ref{prop:orientcoloring}) and
the dual graph of positively decorable simplicial complexes is necessarily bipartite (Proposition \ref{prop:simplicialcomploriented}).
In summary:
$$\text{balanced}\,\Longrightarrow\,
\text{positively decorable}\,\Longrightarrow\,
\text{bipartite}.$$

By results of Joswig \cite{J}, for locally strongly connected simplicial
complexes which are simply connected, these three properties are equivalent.
However, it is not the case for simplicial complexes which do not verify these
assumptions as Figure \ref{fig:noteq} gives an example  of a simplicial complex whose dual
graph is bipartite but which is not balanced. The reader can verify that it is
positively decorable, hence positive decorability is not equivalent to balancedness.

We do not know an example of a pure simplicial complex of dimension $d$
embedded in $\R^d$ whose dual graph is bipartite but which is not positively
decorable, and we therefore ask the following question.

\begin{question}\label{qu:equiv}
  Is a pure full-dimensional simplicial complex positively decorable if and only if its dual graph is bipartite?
 If not, construct a counterexample.
\end{question}

\begin{figure}
  \centering
\begin{tikzpicture}[mynode/.style={circle, inner sep=1.5pt, fill=red}, scale =
  1]
  \node[draw,mynode] (A1) at (0,0) {};
  \node[draw,mynode] (A2) at (0,2) {};
  \node[draw,mynode] (A3) at (1,1) {};
  \node[draw,mynode] (A4) at (2,0) {};
  \node[draw,mynode] (A5) at (3,2) {};
  \node[draw,mynode] (A6) at (3,-1) {};
  \node[draw,mynode] (A7) at (1,-1) {};
  \fill[opacity=0.3, blue] (A1.center) -- (A2.center) -- (A3.center) -- cycle; 
  \fill[opacity=0.3, red] (A2.center) -- (A3.center) -- (A5.center) -- cycle; 
  \fill[opacity=0.3, blue] (A3.center) -- (A4.center) -- (A5.center) -- cycle; 
  \fill[opacity=0.3, red] (A4.center) -- (A6.center) -- (A5.center) -- cycle; 
  \fill[opacity=0.3, blue] (A4.center) -- (A6.center) -- (A7.center) -- cycle; 
  \fill[opacity=0.3, red] (A1.center) -- (A4.center) -- (A7.center) -- cycle; 
  \draw[black] (A1) -- (A2) -- (A5) -- (A6) --
  (A7) -- (A1) -- (A4) -- (A3) -- (A1); 
  \draw[black] (A2) -- (A3) -- (A5);
  \draw[black] (A3) -- (A4);
  \draw[black] (A4) -- (A5);
  \draw[black] (A7) -- (A4) -- (A6);
\end{tikzpicture}
\caption{A two-dimensional simplicial complex whose dual graph is bipartite but
  which is not balanced. The white triangle is not a $2$-simplex of the
  complex.\label{fig:noteq}}
\end{figure}
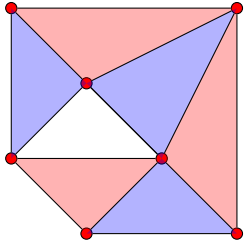

\section{Polytopes and maximally positive systems}\label{sec:maxpos}

Next, we turn our attention to finite sets which are supports of maximally positive
systems. We recall
that maximally positive systems have all their toric complex solutions in the
positive orthant.

\subsection{Order polytopes
} Here we recall the description given in \cite{SS}, which is itself based on
\cite{Sta}.  Let $P$ be any finite partially ordered set, a {\it poset} for
short, on $n$ elements.  A {\it chain} of $P$ is a subset which is totally
ordered. Two subsets $A,B$ of $P$ are called incomparable if for all  $(a,b)
\in A \times B$, the elements $a$ and $b$ are incomparable.  The {\it order
polytope} $\calO (P)$ of $P$ is the set of points $y \in [0,1]^P$ such that
$y_a \leq y_b$ whenever $a \leq b$ in $P$. 

\begin{example}\label{Ex:poset}
\begin{enumerate}
\item If all elements of $P$ are incomparable, then $\calO (P)=[0,1]^P$.
\item If $P$ is a totally ordered set with $d$ elements, then $\calO (P)$ is a primitive $d$-simplex. For instance, if $P=\{1,2,\ldots,n\}$ with usual increasing order, then $\calO (P)$
is the unit simplex with vertices the origin and $\sum_{i=k}^d \mathbf e_i$ for
$k=1,\ldots,d$, where $\mathbf e_i$ stands for the $i$-th vector of the canonical basis.
\item If $P=\{1,2,3\}$ with order given by $1 \leq 2$ and $1 \leq 3$, then $\calO (P)$ is the convex-hull of the set of points $(0,0,0)$, $(0,1,0)$, $(0,0,1)$, $(0,1,1)$, $(1,1,1)$.
\item If $P$ is a disjoint union of incomparable chains of lenghts $d_1,\ldots,d_k$, then the order polytope is isomorphic to a Cartesian product of unit simplices of dimensions $d_1,\ldots,d_k$.
\item More generally, if $P$ is a disjoint union of incomparable sub-posets $P_1,\ldots, P_k$, then $\calO (P) \simeq \prod_{i=1}^k {\calO (P)}_i$.
\end{enumerate}
\end{example}

A {\it linear extension} of $P$ is an order-preserving bijection from $P$ to
$[d]=\{1,2,\ldots,d\}$.  To each linear extension $\lambda$ of $P$, we
associate an unimodular $d$-dimensional simplex $\Gamma_{\lambda} \subset
\calO(P)$ defined by $0 \leq y_{a_1} \leq \cdots \leq y_{a_d} \leq 1$, where
$a_i=\lambda^{-1}(i)$. The vertices of $\Gamma_{\lambda}$ are $(0,\ldots,0)$
and $\sum_{i=k}^d \mathbf e_{\lambda^{-1}(i)}$ for $k=1,\ldots,d$, where
$(\mathbf e_a)_{a \in
P}$ stands for the canonical basis of $\R^P$.  The simplices $\Gamma_{\lambda}$
are the $d$-dimensional simplices of a triangulation $\Gamma(P)$ of $\calO (P)$
called {\it canonical triangulation}. Two $d$-dimensional simplices
$\Gamma_{\lambda}$ and $\Gamma_{\lambda'}$ in $\Gamma(P)$ have a common facet
if and only if there is transposition $\tau$ of $[d]$ such that $\lambda'=\tau
\circ \lambda$.  Fixing a linear extension of $P$ identifies each linear
extension of $P$ with a permutation of $[d]$, where the fixed linear extension
is identified with the identity permutation. The {\it sign} of a simplex
$\Gamma_{\lambda}$ is then defined as the sign of the corresponding
permutation. Note that this is defined up to the choice of a fixed linear
extension, another choice eventually changes simultaneously all signs. Thus,
the adjacency graph of $d$-dimensional simplices of $\Gamma(P)$ is bipartite.
It follows then from \cite[Corollary 11]{J} that $\Gamma(P)$ is balanced. In
fact, this can be shown directly by noting that the map $y \mapsto |y|$, (where
$|y|$ is the number of non-zero coordinates of $y$) restricts to a
$(d+1)$-coloration (with values in $\{0,1,\ldots,d\}$) on $\{0,1\}^P$ giving
different values to a pair of adjacent vertices of $\Gamma(P)$. It turns out
that the canonical triangulation $\Gamma(P)$  is also regular. A convex
function certifying the convexity of $\Gamma(P)$ is given by $\nu(y)=|y|^2$
\cite[Lemma 4.6]{SS}.  We now summarize these properties of $\Gamma(P)$.

\begin{proposition} \label{P:Ordergood}
\cite{SS}
The canonical triangulation of $\calO (P)$ is regular, unimodular, and balanced.
\end{proposition}

\begin{theorem}\label{T:MainForOrder}
For any poset $P$, there exists a real polynomial system with Newton polytope $\calO (P)$ which is maximally positive.
\end{theorem}
\begin{proof}
This follows readily from Theorem \ref {T:manypositive} and Proposition \ref{P:Ordergood}.
\end{proof}

We can be more explicit. As we already saw, the map $(y_a)_{a \in P} \in \calO
(P) \cap \Z^P \mapsto |y|$ is a $(d+1)$-coloring map onto $\{0,1,\ldots,d\}$
giving distinct values to adjacent vertices of $\Gamma(P)$. Then, as shown in
the proof of Proposition \ref{prop:orientcoloring}, the triangulation
$\Gamma(P)$ is decorated by the matrix $C$ with column vector
$-(e_1+\cdots+e_d)$ corresponding to $(0,\ldots,0)$ and column vector $e_{|w|}$
for any other vertex $w$ of $\Gamma(P)$.  We now use the convex function
$y\mapsto |y|^2$, which shows that $\Gamma(P)$ is regular, to get the following
Viro polynomial system.

\begin{equation}\label{E:Viroposet}
  \sum_{w \in {\rm Vert}(\Gamma)\setminus \{(0,\ldots,0)\}}
t^{|w|^2}
e_{|w|}
X^{w}
-(e_1+\cdots+e_d)
=0.
\end{equation}
By Proposition \ref{P:maximally}, for $t>0$ small enough the system \eqref{E:Viroposet} is maximally positive. For instance, if $P=\{1,2,3\}$ with partial order defined by $1 \leq 2$, then
$\calO (P)$ is the convex-hull of the points $(0,0,0)$, $(0,1,0)$, $(1,1,0)$, $(0,0,1)$, $(0,1,1)$, $(1,1,1)$ (a prism) and \eqref{E:Viroposet} may be written as
$t(x_2+x_3)=t^4(x_1x_2+x_2x_3)=t^9x_1x_2x_3=1$.

\medskip

Systems of multidegree $(1,\ldots, 1)$ are a special case of Theorem
\ref{T:MainForOrder}: their support is the order polytope of a disjoint union
of incomparable chains (see Example \ref{Ex:poset}, item 4). This implies the
following statement.

\begin{corollary}
Let $d$ be any positive integer and $(d_1,\ldots,d_k)$ be any partition of $d$
into positive integers. Set $X=(X_1,\ldots,X_k)$, where $X_i=(X_{ij})_{1\leq j
\leq d_i}$ for $i=1,\ldots,k$.
Among multilinear (in other words multidegree $(1,\ldots,1)$) polynomial systems in $X=(X_1,\ldots,X_k)$, there exists a maximally polynomial system.
\end{corollary}

There is another possible construction of maximally positive multilinear
systems, which uses strictly totally positive matrices (\emph{i.e.} matrices
whose all minors of any size is positive). Note that such matrices exist
\cite[Thm 2.7]{And}.  The interest of the following proposition is that it
gives a direct construction of maximally positive multilinear systems which
does not depend on a parameter $t$. However, it does not seem to generalize
easily to multi-homogeneous systems, while the construction above does (we give
a proof of this fact at
the end of this section).

\begin{proposition}
Let $d$ be any positive integer and $d_1+\dots + d_k = d$ be a partition of $d$
into positive integers. Then there exist strictly totally
positive matrices $(T^{(1)},\ldots, T^{(k)})$ of respective dimensions $(d_1+1)\times d,\ldots,
(d_k+1)\times d$ such that all toric complex solutions of the system $f_1(X) =
\dots = f_d(X) = 0$ lie in the positive orthant, where
$$f_i=\left((-1)^{d_1+1}T^{(1)}_{d_1+1,i}+\sum_{j=1}^{d_1} (-1)^j T^{(1)}_{j,i}X_{1j}\right)\times\dots\times\left((-1)^{d_k+1}T^{(k)}_{d_k+1,i}+\sum_{j=1}^{d_k} (-1)^j T^{(k)}_{j,i}X_{kj}\right).$$
\end{proposition}

\begin{proof} By Kouchnirenko's theorem, the number of complex toric
  non-degenerate solutions is bounded by the multinomial coefficient
  $\binom{d}{d_1,\ldots,d_k}$. 
For any partition of the set $\{1,\ldots,d\}$ into $k$ parts $E_1\cup\dots\cup
E_k$ of respective sizes $n_1,\ldots, n_k$, we associate the linear system
$\ell_1(X) = \dots=\ell_d(X) = 0$, where $$ \ell_i (X) =
(-1)^{d_u+1}T^{(u)}_{d_u+1,i}+\sum_{j=1}^{d_u} (-1)^j T^{(u)}_{j,i}X_{1j},
\text{ where $u\in\N$ is such that $i\in E_u$.}$$ By construction, $\ell_i(X)$
divides $f_i(X)$, hence a solution of the linear system is a solution of
$f_1(X) = \dots = f_d(X) = 0$. Next, note that the condition that the matrices
$T$ are totally positive imply that the linear system $\ell_1(X) =
\dots=\ell_d(X) = 0$ has a unique solution in the positive orthant. There are
$\binom{d}{d_1,\ldots,d_k}$ possible partitions of $\{1,\ldots,d\}$ into $k$
parts $E_1\cup\dots\cup E_k$ of respective sizes $d_1,\ldots, d_k$. Each of
them yields one solution in the positive orthant. 
To finish the proof, we
prove that there exist $(T^{(1)},\ldots, T^{(k)})$ such that all these solutions are distinct. This is done by noticing that the set of $t$-uples of
strictly totally positive matrices is a non-empty open subset for the Euclidean
topology of $t$-uples of matrices with real entries, while the set of such
tuples of matrices leading to coalescing solutions is Zariski closed and hence
has Lebesgue measure $0$. 
\end{proof}

The vertices of $\calO (P)$ are characteristic functions of upper order ideals of $P$.
The set of such upper order ideals ordered by inclusion is a distributive lattice.
The toric ideal of the set of integer points of the order polytope $\calO (P)$
is generated by binomials $x_Jx_K-x_{J \wedge K}x_{J \vee K}$ over all
incomparable upper order ideals $J,K$ of $P$, where $J \wedge K=J \cap K$ and
$J \vee K=J \cup K$ \cite{H}. These binomials are homogeneous binomials with
exponents at most~$2$ (in fact at most~$1$). Thus Bihan's conjecture
(see the introduction) holds true for order polytopes.

\bigskip

We finish this section by reporting on other classical families of
polytopes which admit unimodular balanced regular triangulations.
We stress that in all the following cases, Bihan's conjecture \cite{B} holds.

\medskip

{\bf Cross polytopes.} The $d$-dimensional cross polytope is the subset of $\R^d$ defined by
points verifying $\lvert x_1\rvert+\dots +\lvert x_d\rvert\leq 1$. It
has normalized volume $2^d$ and has a regular unimodular triangulation
obtained by slicing it along the coordinate hyperplanes $x_i=0$ for
$i\in\{1,\ldots, d\}$. Associating a sign to each of the simplex of the
triangulation counting the parity of the number of negative
coordinates provides a $2$-coloring of its dual graph. By Proposition~\ref{prop:orientcoloring} this triangulation is balanced. 

\medskip

{\bf Products of balanced triangulations.} If $\mathcal A_1$, $\mathcal
A_2$ are two point configurations which admit regular balanced unimodular
triangulations, then the product of these triangulations yields a regular
balanced subdivision of the convex hull of $\mathcal A_1\times\mathcal A_2$
into products of simplices. Joswig and Witte show in \cite{JW} that this
subdivision can be refined into a regular balanced unimodular triangulation.
Consequently, if $\mathcal A_1$ and $\mathcal A_2$ admit regular balanced and unimodular triangulation, then so does $\mathcal A_1\times \mathcal A_2$.
\medskip

{\bf Joins of balanced triangulations.}
If $P_1 \subset \R^{d_1}$, $P_2 \subset \R^{d_2}$ are two full-dimensional polytopes admitting regular unimodular balanced triangulations, then the natural triangulation of the join $P_ 1 \star  P_2
\subset \R^{d_1+d_2+1}$ by joins $\sigma_1 \star \sigma_2$ of full-dimensional simplices in the triangulations of $P_1$ and $P_2$ is regular and unimodular
(see \cite[Section 2.3.2]{HPPS}). The fact that this triangulation is balanced can be seen by coloring the vertices of the triangulation of
$P_1$ by $\{0,\ldots,d_1\}$ and the vertices of the triangulation of $P_2$ by $\{d_1+1,\ldots, d_1+d_2+2\}$.
\medskip

{\bf Alcoved polytopes.} Alcoved polytopes are polytopes whose codimension
$1$ faces lie on hyperplanes of the form $x_i - x_j = \ell$, with $\ell\in\N$.
Lam and Postnikov \cite{LamPos} showed that such polytopes admit a natural
unimodular triangulation compatible with the subdivision of $\R^d$ given by the
complements of all the hyperplanes of the form $x_i - x_j = \ell$, for $\ell\in\N$. 
The fact that this triangulation of alcoved polytopes is regular and balanced is a special case of Lemma \ref{L:regbal} below.

Let $\{b_1,\ldots,b_s\} \subset \Z^d$ be a collection of vectors.
This induces an infinite arrangement $\mathcal H$ of hyperplanes
$\{x \in \R^d \, | \, \langle b_i, x \rangle=k \}$
for all $i=1,\ldots, s$ and $k \in \Z$.
Let $Q$ be a $d$-dimensional lattice polytope in $\R^d$. Assume that each
$(d-1)$-dimensional face of $Q$ is contained in some hyperplane of~$\mathcal
H$. Consider the subdivision $\Gamma$ of $Q$ obtained by slicing $Q$ along the
hyperplanes of $\mathcal H$.

\begin{lemma}\label{L:regbal}
The subdivision $\Gamma$ is regular and its dual graph is bipartite.
\end{lemma}
\begin{proof}
The regularity of $\Gamma$ is obtained by considering the restriction to the vertices of $\Gamma$ of $f(x)= \sum_{i=1}^s \langle b_i, x \rangle^2$ (see \cite[Theorem 2.4]{HPPS}).
It remains to show that the dual graph of $\Gamma$ is bipartite. In fact we show that the dual graph of the subdivision of $\R^n$ obtained by slicing along the hyperplanes of $\mathcal H$ is bipartite. For any given $i=1,\ldots,s$, associate a sign to each connected component of
$$\R^d \setminus \cup_{k \in \Z}
\{x \in \R^d \, | \,  \langle b_i, x \rangle =k \}$$ so
that adjacent components get opposite signs.
Any connected component of $\R^d \setminus {\mathcal H}$ is a common
intersection of such connected components for $i=1,\ldots,s$, and we equip it
with the product of the corresponding signs. Clearly, two adjacents components
of $\R^d \setminus {\mathcal H}$ have opposite signs.
\end{proof}

\medskip

{\bf The hypersimplex.} For $d,k \in\N$, $k\leq d$, the hypersimplex
$\Delta_{d,k}$ is the convex hull of all vectors in $\R^d$ whose
coordinates are $k$ ones and $d-k$ zeros. Several unimodular
triangulations of the hypersimplex have been proposed. One is given by
Sturmfels in \cite{St2}. Another one is described
by Stanley in \cite{Sta2}. As explained in \cite[Section 2.3]{LamPos}, the
hypersimplex is linearly equivalent to an alcoved polytope. Consequently, it admits a unimodular balanced and regular triangulation.

\medskip

{\bf Multi-homogeneous systems.} A multi-homogeneous system with respect to a
partition $d_1+\dots+d_k = d$ and with multi-degrees $(\ell_1,\ldots, \ell_k)$
is a system whose monomial support correspond to the set of lattice points in
the polytope $\ell_1\,\Delta_{d_1}\times \dots \times \ell_k\,\Delta_{d_k}$,
where $\Delta_{d_i}$ is the convex hull of $\{\mathbf 0, \mathbf e_1,\ldots,
\mathbf e_{d_i}\}$ in $\R^{d_i}$.
Using the construction by Joswig and Witte \cite{JW}, it is therefore
sufficient to show that $\ell\,\Delta_d \subset \R^d$ admits a regular balanced
unimodular triangulation for any $d, \ell\in\N$. Under the linear change of variables
$z_i = x_1+\cdots+x_i$, we see that $\ell\,\Delta_d \subset \R^d$ is defined by inequalities
$z_1 - z_0 \geq 0$, $z_{i+1} - z_i \geq 0$ for $i=1,\ldots,d-1$,  $z_d - z_0 \leq \ell$
and $z_0 = 0$. Thus as explained in \cite[Section 2.3]{LamPos}, slicing $\R^{d+1}$ by hyperplanes
$z_j-z_i=k$ for $1 \leq i <j \leq d$ and $k \in \Z$ (together with $z_0=0$) gives a unimodular triangulation of $\ell\,\Delta_d \subset \R^d$.
This triangulation is regular and balanced by Lemma \ref{L:regbal}.
Itenberg and Viro constructed another interesting regular unimodular triangulation of $\ell\,\Delta_d$ by induction on the dimension $d$ \cite[Section 5]{IV}.
This triangulation is obtained by taking joins of dilates of the simplex $\Delta_{d-1}$. Using the property that the natural triangulation
of a join of polytopes equipped with regular, unimodular and balanced triangulations is also regular, unimodular and balanced, one can easily show inductively
that the triangulation used in \cite{IV} is balanced.

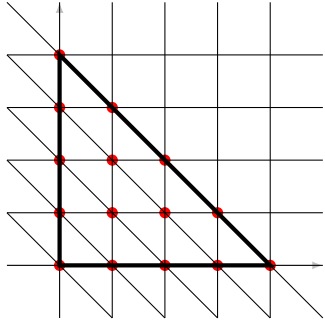
\begin{figure}\centering
\begin{tikzpicture}[mynode/.style={circle, inner sep=1.5pt,
  fill=red}, scale = 0.7]
    \coordinate (Origin)   at (0,0);
    \coordinate (XAxisMin) at (-1,0);
    \coordinate (XAxisMax) at (5,0);
    \coordinate (YAxisMin) at (0,-1);
    \coordinate (YAxisMax) at (0,5);
    \draw [thin, lightgray, -latex] (XAxisMin) -- (XAxisMax);
    \draw [thin, lightgray,-latex] (YAxisMin) -- (YAxisMax);
  \foreach \ai in {0,...,4} 
     \foreach \aj in {\ai,...,4}
     {\draw (\ai,4-\aj) node[mynode] {};}
  \draw[ultra thick] (0,0) -- (0,4) -- (4,0) -- cycle;
  \foreach \a in {0,...,4}
  {\draw[thin] (\a, -1) -- (\a, 5);
    \draw[thin] (-1, \a) -- (5, \a);
    \draw[thin] (\a+1,-1) -- (-1,\a+1);
 }

\end{tikzpicture}
\caption{A regular balanced unimodular triangulation of
  $4\,\Delta_2$\label{fig:triangd2l4}}
\end{figure}

\section{Fewnomial systems with many positive solutions}\label{sec:cyclic}

In this section, we consider the following classical problem.
\begin{problem}
Let $d,k\in\N$ be integers. Let $\Xi_{d,k}$ be the maximum number of
non-degenerate solutions in the positive orthant over all polynomial
systems of $d$ equations in $d$ variables with real coefficients involving at most $d+k+1$ monomials. Find lower
and upper bounds on $\Xi_{d,k}$.
\end{problem}

Khovanskii \cite{Khov} proved that $\Xi_{d,k}$ is bounded from above by a
function of $d$ and $k$, and he proved that
$\Xi_{d,k}\leq 2^{\binom{d+k}2}(d+1)^{d+k}$ \cite{Khov}. This bound was later improved by Bihan and Sottile \cite{BS} to
$$\Xi_{d,k}\leq \dfrac{e^2+3}{4} 2^{\binom{k}2}d^k.$$
This is currently the best known upper bound. Bihan, Rojas and Sottile \cite{BRS} proved the following lower bound:
$$(\lfloor d/k\rfloor+1)^k\leq \Xi_{d,k},$$
or more generally, if $d_1+\dots+d_k=d$ is a partition of $d$, then $\prod_{1\leq i\leq k}(d_i+1)\leq \Xi_{d,k}$.

A system consisting of $d$ dense univariate polynomials of degree $\lfloor
k/d\rfloor+1$ proves the inequality:
$$(\lfloor k/d\rfloor+1)^d\leq \Xi_{d,k}.$$
More generally, if $k_1+\dots+k_d=k$ is a partition of $k$, then $\prod_{1\leq i\leq\ell}(k_i+1)\leq\Xi_{d,k}$.

It is natural to study the sharpness of these two last bounds. The
first one is sharp when $k=1$ by \cite{B}. The second bound is sharp when
$d=1$ by Descartes' rule of signs. When $k=d$, these two bounds give $2^d\leq \Xi_{d,d}$.

In order to construct systems with many positive solutions, we now turn our
attention to subsimplicial complexes of triangulations of
cyclic polytopes. These triangulations are good candidates since they enjoy a large
number of simplices compared to their number of vertices.

\begin{definition}
  Let $a_1<a_2<\dots<a_n$ be $n$ real numbers. For $d,n\in\N$, the cyclic polytope $C(n,d)$ associated to $(a_1,\ldots, a_n)$ is the convex hull in $\R^d$ of the points $\{(a_i,a_i^2,\ldots, a_i^d)\}_{1\leq i\leq n}$. The combinatorial structure of $C(n,d)$ is independent of the choice of the real numbers $a_1,\ldots, a_n$.
\end{definition}

Projecting the lower part of $C(n,d+1)$ with respect to the last coordinate onto the first $d$ coordinates yields a regular triangulation of $C(n,d)$.
This triangulation denoted by $\hat{\mathbf 0}_{n,d}$ is the minimum element
for the \emph{higher Stasheff-Tamari orders} on $C(n,d)$ \cite{Ede}.
We now recall the description of $\hat{\mathbf 0}_{n,d}$ given in \cite[Lemma
2.3]{Ede} and provide a proof for the reader's convenience.
In what follows, we represent a $d$-simplex of $\hat{\mathbf 0}_{n,d}$
by the increasing sequence of $d+1$ integer numbers between $1$ and $n$ which labels
its vertices. 

\begin{proposition}\label{prop:cyclic}
If $d$ is odd, the $d$-simplices of $\hat{\mathbf 0}_{n,d}$ are 
$$1 \leq i_{1} < i_{1}+1 <i_{2} < i_{2}+1< \cdots <i_{\frac{d+1}{2}} < i_{\frac{d+1}{2}}+1 \leq n,$$ while for even $d$ they are the simplices
$$1 = i_{1} <i_{2} < i_{2}+1< \cdots <i_{\frac{d}{2}} < i_{\frac{d}{2}}+1 \leq n.$$
The triangulation $\hat{\mathbf 0}_{n,d}$ is regular, and its number of
$d$-simplices is
$$\begin{cases}
\binom{n-(d+1)/2}{(d+1)/2}\text{ if $d$ is odd} \\
\\
\binom{n-1-d/2}{d/2}\text{ if $d$ is even.}
\end{cases}$$
\end{proposition}
\begin{proof}
In the following, we show that $\hat{\mathbf 0}_{n,d}$ is the regular triangulation of $C(n,d)$ associated with the height function $(a_i,a_i^2,\ldots, a_i^d) \mapsto P(a_i)$, where $P(x)=x^{d+1}$.
All monic univariate polynomials $P$ of degree
$d+1$ give rise to  height functions defining the same triangulation $\hat{\mathbf 0}_{n,d}$.
Indeed, if $P_1$ and $P_2$ are two such polynomials, then $P_1-P_2$ is a polynomial of degree $d$ and can therefore be expressed as an affine function on the rational normal
curve of degree~$d$. 
Let $1 \leq j_1<\cdots<j_{d+1} \leq n$ be any maximal simplex of $\hat{\mathbf 0}_{n,d}$. The intervals of $\R$ where the univariate polynomial
$P(T)=(T-a_{j_1})\cdots(T-a_{j_{d+1}})$ is nonnegative is the union
$$\begin{cases}
  [a_{j_1},a_{j_2}]\,\cup\,[a_{j_3},a_{j_4}]\,\cup\,\dots\cup\,[a_{j_d},a_{j_{d+1}}]\text{\quad if $d$ is odd, or}\\
  ]-\infty,a_{j_1}]\,\cup\,[a_{j_2},a_{j_3}]\,\cup\,\dots\cup\,[a_{j_{d}},a_{j_{d+1}}]\text{\quad if $d$ is even.}
\end{cases}$$
The condition that $P$ is nonnegative on $\{a_1,\ldots, a_n\}$ is equivalent to
saying that none of the $a_i$ belongs to these intervals. This implies that
every 
simplex has the form claimed in Proposition \ref{prop:cyclic}. For instance, if $d$ is even, it implies that
$a_1$ is a vertex of all the $d$-simplices. Since $P$ is positive
on $\{a_1,\ldots, a_n\}\setminus\{a_{j_1},\ldots, a_{j_{d+1}}\}$, the simplex
$1 \leq j_1<\ldots<j_{d+1} \leq n$ belongs to the regular triangulation induced
by the polynomial $P(T) =(T-a_{j_1})\cdots(T-a_{j_{d+1}})$. But we showed that
this triangulation does not depend on the choice of the monic polynomial $P$
and is therefore $\hat{\mathbf 0}_{n,d}$.
\end{proof}

Unfortunately, the triangulation $\hat{\mathbf 0}_{n,d}$ cannot be positively
decorated (except in trivial cases) since its dual graph is not bipartite.

\begin{example}\label{E:O63}
The triangulation $\hat{\mathbf 0}_{6,3}$ of the cyclic
polytope $C(6,3)$ consists of the following tetrahedra (together with their faces):
$A=1 <2 < 3 < 4$, $B=1< 2 < 4 < 5$, $C= 1 <2 < 5< 6$, $D=2 < 3 < 4<5$
$E=2 < 3 < 5 < 6$, $F= 3 < 4< 5 < 6$. Its dual graph is shown in Figure \ref{F:dualgraph}.
It is not bipartite since it contains cycles of length $3$.
\end{example}

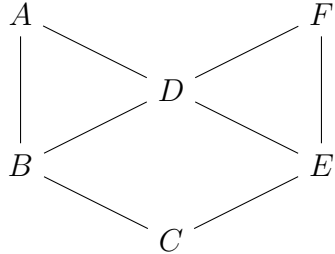
\begin{figure}[htb]
\begin{center}
  \begin{tikzpicture}[scale=2]
   \node (A) at (-1, 1.5) {$A$};
   \node (B) at (-1,0.5) {$B$};
   \node (C) at (0,0) {$C$};
   \node (D) at (0,1) {$D$};
   \node (E) at (1,0.5) {$E$};
   \node (F) at (1,1.5) {$F$};
\draw (A) -- (B) -- (C) -- (E) -- (F) -- (D) -- (A);
\draw (B) -- (D) -- (E);
  \end{tikzpicture}
\caption{Dual graph of $\hat{\mathbf 0}_{6,3}$.}  
\label{F:dualgraph}
\end{center}
\end{figure}
\bigskip

We now restrict our attention to the case where $d$ is odd. We represent
any $d$-simplex $1 \leq i_{1} < i_{1}+1 <i_{2} < i_{2}+1< \cdots <i_{\frac{d+1}{2}} < i_{\frac{d+1}{2}}+1 \leq n$ of $\hat{\mathbf 0}_{n,d}$ by the sequence $1\leq i_1<\dots< i_{\frac{d+1}{2}}\leq n-1$. Note that $i_j - i_{j+1}\geq 2$ for $j=1,\ldots, \frac{d+1}{2}$. Consider the subsimplicial $d$-dimensional complex ${\mathbf S}_{n,d}$ of $\hat{\mathbf 0}_{n,d}$ whose maximal simplices are the $d$-simplices
$1\leq i_1<\dots< i_{\frac{d+1}{2}}\leq n-1$ such that for all $j$, either $i_j$ is odd, or $i_{j+1} - i_j > 2$.

\begin{proposition}\label{prop:cyclicbipartite}
For $d$ odd, the dual graph of the simplicial complex ${\mathbf S}_{n,d}$ is
bipartite and the number of $d$-simplices of ${\mathbf S}_{n,d}$ equals the coefficient of $X^n Y^{(d+1)/2}$ in the series expansion of the rational function
$$ \dfrac{1+X+X^3 Y}{1-X^2-X^2 Y-X^4 Y}.$$
\end{proposition}

\begin{proof} Two simplices $1\leq
i_1<\dots< i_{(d+1)/2}\leq n$ and $1\leq i'_1<\dots< i'_{\frac{d+1}{2}}\leq
n$ of ${\mathbf S}_{n,d}$ are adjacent if and only if the two sets $\{i_1, i_1+1,i_2,
i_2+1,\ldots, i_{\frac{d+1}{2}},i_{\frac{d+1}{2}}+1\}$ and $\{i'_1, i'_1+1,i'_2,
i'_2+1,\ldots, i'_{\frac{d+1}{2}},i'_{\frac{d+1}{2}}+1\}$ share $d$ elements in
common. This implies that the symmetric difference of the sets
$\{i_1,\dots, i_{\frac{d+1}{2}}\}$ and $\{i'_1,\dots, i'_{\frac{d+1}{2}}\}$ is an
interval $I=\{k+1,\ldots, k+2\ell\}$ for some integers $k$ and $\ell$. Since
$i_j$ is odd, or $i_{j+1}-i_j>2$, the cardinality of $I$ must equal two. Hence, 
in ${\mathbf S}_{n,d}$, the parity of $i_1+\dots+i_{\frac{d+1}{2}}$ alternates
between adjacent $d$-simplices. Consequently, the dual graph is bipartite.

Let $c_{n,d}$ denote the number of $d$-simplices in $\mathbf S_{n,d}$. It is
easily verified that $c_{n,d}$ verifies the recurrence relation $c_{n,d} =
c_{n-2,d}+ c_{n-2,d-2} + c_{n-4, d-2}$. Let $G(X,Y) = \sum_{i,j\geq 0} c_{n,d}
X^n Y^d$ be the associated generating series. The linear
recurrence and the initial conditions imply that $(1-X^2-X^2Y^2-X^4Y^2) G(X,Y) =
(1+X+X^3 Y^3)$. Consequently, $c_{n,d}$ equals the coefficient of $X^n Y^d$ in
the series expansion of $$G(X,Y) = \dfrac{Y+XY+X^3 Y^3}{1-X^2-X^2Y^2-X^4Y^2}.$$
\end{proof}

The simplicial complex ${\mathbf S}_{n,d}$ is bipartite but it is not balanced
in general, as demonstrated by the following example.

\begin{example}\label{E:S63}
The simplicial complex ${\mathbf S}_{6,3}$ is obtained by removing
$D$ from  $\hat{\mathbf 0}_{6,3}$, see Example \ref{E:O63}.
It is easily seen that this simplicial complex is not balanced although its dual graph is bipartite.
In fact the star of the vertex $4$ contains the tetrahedra $A$, $B$, and $F$
and its star is thus not connected. Consequently, the simplicial complex is not locally strongly connected. 
\end{example}

\begin{corollary}\label{coro:bivseries}
For integers $i,j\in\N$ and a rational function $S\in\Z(X,Y)$, we let
$[X^i Y^j]S(X,Y)$ denote the coefficient of $X^i Y^j$ in the formal
series expansion of $S$. If the simplicial complex $\mathbf S_{2 d+1,d}$
 is
positively decorable, then for $d$ odd,
$$[X^{2 d +1} Y^{(d+1)/2}]\dfrac{1+X+X^3 Y}{1-X^2-X^2 Y - X^4 Y}\leq \Xi_{d,d}.$$
\end{corollary}

The best known lower bound on $\Xi_{d,d}$ is $2^d\leq \Xi_{d,d}$. 
The following proposition shows that if $\mathbf S_{2 d+1,d}$ is positively
decorable,
then we would obtain the following sharper lower bound on $\Xi_{d,d}$.

\begin{proposition}
  For $d$ odd, let $c_d$ denote the number of $d$-simplices in $\mathbf S_{2 d +1, d}$. As
  $d$ grows, 
  $$c_d \sim \dfrac{(\sqrt{2}+1)^d}{\sqrt{d}}\,\,\dfrac{2^{1/4} (1+\alpha)}{4 \alpha\sqrt{\pi}},$$
  where $\alpha=3-2\sqrt{2}$. Consequently, $\lim c_d^{1/d} = \sqrt{2}+1.$
\end{proposition}

\begin{proof}
For a bivariate series $G(X,Y)$, let $H(X,Y)$ and $U(X,Y)$ be the series such that 
$H(X^2,Y)=(X G(X)-XG(-X))/2$ and $U(X^2,Y)=(H(X,Y)+H(-X,Y))/2$. Note that, if
we set $k=(d+1)/2$, then for $d$ odd, we have 
$$[X^k Y^k] U(X,Y) = [X^{2 d+1}Y^{(d+1)/2}] G(X,Y).$$ 

If $G$ is the series in Corollary \ref{coro:bivseries}, then $U(X,Y)$ equals
$$U(X,Y)=\dfrac{X (X Y^2-2 Y-1)}{1+X^2 Y^2+(-Y^2-4 Y-1)X}.$$

The diagonal series $D(Z)=\sum_{i\geq 0} ([X^i
Y^i] U(X,Y)) Z^i$ is obtained by considering $U(X/Y,Y)/Y$ as a univariate
function
with coefficients in $\R(X)$ and by examinating its poles which tend to zero as $X$ goes to
zero, see \emph{e.g.} \cite[Section 4]{BDS} and references therein.
In this setting, $U(X/Y,Y)/Y$ has three poles:
$$0,\quad\dfrac{X^2-4 X+1 - \sqrt{(X-1)^2 (X^2 -6 X+1)}}{2 X},\quad\dfrac{X^2-4
  X+1 + \sqrt{(X-1)^2 (X^2 -6 X+1)}}{2 X}.$$
Only the first two poles tend to zero as $X$ goes to zero, and the sum of their
residues gives the desired algebraic series
$$D(X) = \frac 12\left(\dfrac{1+X}{\sqrt{X^2-6 X +1}} - 1\right).$$
The series $D(X)$ is analytic at $0$.
By \cite[Theorem IV.7]{FS}, the asymptotic exponential behavior of the
coefficients of the
series expansion is then dictated by its complex singularity nearest to the origin.
In our case, the singularity of $D(X)$ with smallest complex module is
$\alpha = 3-2\sqrt{2}$. Consequently, we obtain that
$$\limsup_{k\rightarrow\infty} \left([X^k Y^k] U(X,Y)\right)^{1/k} =
\frac{1}{3-2\sqrt{2}}=\left(\sqrt 2 +1\right)^2.$$
Using the relation $k=(n+1)/2$ proves the second statement of the proposition.
In order the prove the first statement, we need to look more precisely at the
nature of the singularity of $D(X)$ at $X=3-2\sqrt{2}$. First, we note that
$\sqrt{\alpha-X}\, D(X)$ is analytic at $\alpha$, hence
$$D(X)=\dfrac{1}{\sqrt{\alpha-X}}\sum_{i\in\N} a_i \, (\alpha-X)^i,$$
for some real values $a_i\in\R$.
Next, using a Taylor expansion around $0$, we get for all $i\in\N$: 
$$[X^k]\dfrac{(\alpha-X)^i}{\sqrt{\alpha -X}} = \dfrac{(2k - 2i -1)!!}{2^k\,
k!}\underset{k\rightarrow\infty}{\sim} \dfrac{1}{\alpha^k\,(2 k)^{i}\sqrt{\alpha \pi k}}.$$
Evaluating
$(\alpha-X)\, D(X)$ at $0$ provides us with the value of $a_0$. Finally we get
$$[X^k]D(X)\underset{k\rightarrow\infty}{\sim}\dfrac{1+\alpha}{4\,2^{1/4}}\,\,\dfrac{1}{\alpha^k\sqrt{\alpha\pi
k}}$$
and using $c_d=[X^{(d+1)/2}] D(X)$ concludes the proof.\end{proof}

We report in Table \ref{table:compar} the number of simplices in $\mathbf
S_{2d+1, d}$ for the first odd values of $d$.

\begin{table}
\centering
\begin{tabular}{|@{\quad\quad}c@{\quad\quad}||r|r|}
\hline
$d$ & $2^d$ & $c_d$ \\
\hline
1& 2& 2\\ 
3& 8& 8\\
5& 32& 38\\
7& 128& 192\\
9& 512& 1002\\
11& 2048& 5336\\
13& 8192& 28814\\
15&32768&157184\\
17& 131072& 864146\\
19& 524288& 4780008\\
21& 2097152& 26572086\\
\hline
\end{tabular}
\caption{For $d$ odd, comparison between $2^d$, the best known lower bound on
  $\Xi_{d,d}$ with the number of simplices $c_d$ in the simplicial complex
  $\mathbf S_{2d+1,d}$. \label{table:compar}}
\end{table}

\begin{example}\label{ex:C63}
The simplicial complex for the cyclic polytope $C(6,3)$ is
positively decorable. Ordering the vertices with respect to their relative
position on the rational normal curve of degree $3$, a coefficient matrix which
decorates the simplicial complex is
$$
\begin{bmatrix}
1 &0&3&-4& 0 &-1\\
-2&1&1&0 & 0 &-1\\
0 &0&3&-3& 1 &-3
\end{bmatrix}
$$

By Theorem \ref{thm:nbpossols}, this implies that for $t>0$ sufficiently small, the polynomial system
$$\begin{array}{rcl}
  1+3 t^{2^4} X^2Y^{2^2}Z^{2^3} - 4 t^{3^4} X^3Y^{3^2}Z^{3^3} - t^{5^4} X^5Y^{5^2}Z^{5^3} &=& 0\\
  -2+t X Y Z + t^{2^4} X^2Y^{2^2}Z^{2^3} - t^{5^4} X^5Y^{5^2}Z^{5^3} &=& 0\\
   3 t^{2^4} X^2Y^{2^2}Z^{2^3} - 3 t^{3^4} X^3Y^{3^2}Z^{3^3} + t^{4^4}
   X^4Y^{4^2}Z^{4^3}- 3 t^{5^4} X^5Y^{5^2}Z^{5^3} &=& 0
\end{array}$$
has at least $5$ real solutions in the positive orthant. Consequently,
$5\leq \Xi_{3,2}$.
\end{example}

We end this section by a special case of Question \ref{qu:equiv}: 

\begin{question}
  For $d$ odd, is the bipartite simplicial complex $\mathbf S_{2d+1,d}$
  always positively decorable?
\end{question}

We verified that it is the case for $d=1,3,5$.
A general positive answer to this question would imply that
$\underset{d\rightarrow\infty}\limsup\, \left(\Xi_{d,d}\right)^{1/d}\geq \sqrt 2+1$.

\section{Realizable oriented matroids and positive matrix
completion}\label{sec:comput}

In this section, we study the problem of decorating positively a simplicial complex from a
computational viewpoint and we exhibit a connection to the problems
of the realizability of oriented matroids and low-rank matrix completion problem.
We start by the following characterization of oriented matrices (compare with Proposition \ref{P:invariant}).

\begin{proposition} \label{P:invariantbis}
Let $M$ be a full rank $d\times (d+1)$ matrix with real entries. The following statements are equivalent:
\begin{enumerate}
\item $M$ is an oriented matrix;
\item If $v_1,\ldots,v_{d+1}$ are the columns vectors of $M$, then the common intersection of half spaces $H_i^+=\{x \in \R^d \; , \; \langle v_i, x \rangle \geq 0\}$
for $i=1,\ldots,d+1$ is reduced to the zero vector.
\end{enumerate}
\end{proposition}
\begin{proof}
We use Proposition \ref{P:invariant}. Assume $M$ is oriented. Then by Proposition \ref{P:invariant} there exist positive real numbers $x_i$ such that $\sum_{i=1}^{d+1} x_i v_i=0$.
Thus if $\delta \in \cap_{i=1}^{d+1} H_i^+$, then $\langle v_i, \delta \rangle =0$ for $i=1,\ldots,d+1$, which gives $\delta=0$ since $v_1,\ldots,v_{d+1}$ generate $\R^d$.
Conversely, assume that $\cap_{i=1}^{d+1} H_i^+ =\{0\}$. Since $M$ has maximal rank, we may assume
that $v_1,\ldots,v_d$ form a basis of $\R^d$ permuting the columns of $M$ if necessary.
Take a dual basis, that is, vectors $u_1,\ldots,u_d \in \R^d$ such that $\langle u_i, v_j \rangle=1$ if $i=j$ and $0$ otherwise.
Note that $u_1,\ldots,u_d \in \cap_{i=1}^{d} H_i^+$.
Write $v_{d+1}=\sum_{i=1}^d\lambda_i v_i$ with $\lambda_1,\ldots,\lambda_d \in \R$. If some coefficient $\lambda_i$ is nonnegative, then $\langle u_i, v_{d+1} \rangle \geq 0$, and thus $u_i$ is a non-zero vector in $\cap_{i=1}^{d+1} H_i^+$. It follows that $\lambda_1,\ldots,\lambda_d < 0$, which implies that $M$ is oriented by Proposition \ref{P:invariant}.
 \end{proof}

A $d\times n$ matrix $C$ with columns vectors $v_1,\ldots,v_n \in \R^d$ determines an oriented matroid on $\{1,\ldots,n\}$ with chirotope given by the signs of the maximal minors of $C$. Consider the hyperplane arrangement $\mathcal{H}=\{H_1,\ldots,H_n\}$, where
$H_i=\{x \in \R^d \, , \, \langle v_i,x \rangle =0\}$. Each connected component of the complementary part (chamber for short)  gets a {\it sign vector} $s=(s_1,\ldots,s_n) \in \{\pm 1\}^n$
recording the signs of the linear forms $\langle v_i, \cdot  \rangle$ on that chamber (where as usual $s_i=1$ means that $\langle v_i, \cdot  \rangle$ is nonnegative on the considered chamber).
It turns out that these sign vectors are precisely the covectors of the oriented matroid determined by $C$, see \cite[p. 11]{BLSWZ}.

\begin{proposition}\label{prop:covectors}
A simplicial complex of dimension $d$ with vertices indexed by $\{1,\ldots,n\}$
can be positively decorated if and only if there exists a realizable oriented
matroid of rank $d$ on $\{1,\ldots,n\}$
such that for any $d$-simplex $j_1<\cdots<j_{d+1}$ of the simplicial complex, there is no covector $s \in \{\pm 1\}^n$ of the oriented matroid
which satisfies $s_{i_1}=\cdots=s_{i_{d+1}}=1$.
\end{proposition}
\begin{proof}
This follows from Proposition \ref{P:invariantbis}.
\end{proof}

Proposition \ref{prop:covectors} shows that computational techniques for
classifying oriented matroids may yield a solution to the problem of
decorating simplicial complexes. Another option is to rely on techniques for
low-rank matrix completion with positivity constraints. The following positive variant of
low-rank matrix completion appears in several applicative problems in
compressed sensing, see \emph{e.g.} \cite{The}.

\begin{problem}[Positive matrix completion problem]
Let $p,q,r\in\N$ be three integers with $r\leq \min(p,q)$, and $M$ be a
$p\times q$ non-negative real matrix with missing entries. Complete the matrix with
\emph{positive} real numbers, such that the completed matrix has rank
$r$.
\end{problem}

The next proposition shows that positively decorating a simplicial complex can
be done by solving a positive matrix completion problem.

\begin{proposition}
Let $\Gamma$ be a pure simplicial complex of dimension $d$ on $n$ vertices,
with $d$-simplices $\tau_1,\ldots,\tau_\ell$.
Let $M$ be a $n\times \ell$ nonnegative matrix such that $M_{i,j}>0$ if
and only if $j$ is a vertex of $\tau_i$.
If $M$ has rank $n-d$, then any full rank $d\times n$ matrix $C$ such that $C\cdot M=\mathbf 0$
positively decorates
$\Gamma$.
\end{proposition}

\begin{proof}
  By construction, each $d\times (d+1)$ submatrix of $C$ corresponding to a
  $d$-simplex of $\Gamma$ has 
  a positive vector in its kernel and is therefore oriented.
\end{proof}

\begin{example}
  We continue our running example \ref{ex:C63}, which corresponds to a bipartite
  simplicial complex on $6$ vertices. It has $5$ $d$-simplices, hence decorating this
  complex is equivalent to finding a rank $3$ matrix of size $6\times 5$ which
  has nonzero entries at positions prescribed by the complex. A solution to
  this problem is

  $$\begin{blockarray}{cccccc}
        &A&B&C&E&F\\
	\begin{block}{c[ccccc]}
    	1 & 1 & 4 & 1 & 0 & 0 \\
	2 & 1 & 8 & 3 & 2 & 0 \\
	3 & 3 & 0 & 0 & 3 & 6 \\
	4 & 4 & 4 & 0 & 0 & 4 \\
	5 & 0 & 3 & 3 & 6 & 3 \\
	6 & 0 & 0 & 1 & 3 & 2\\ 
      \end{block}
  \end{blockarray}
  \quad\quad\quad\quad\begin{array}{rcl}
    A&=&[1234]\\
    B&=&[1245]\\
    C&=&[1256]\\
    E&=&[2356]\\
    F&=&[3456]\end{array}$$
  The matrix in Example \ref{ex:C63} is a basis of the left kernel of this
  matrix.
\end{example}

Using this reformulation of the initial problem, we used the software {\tt
NewtonSLRA} \cite{SpSc} to solve the matrix completion problem and to compute a
decoration of $\mathbf S_{11,5}$. This shows that $\Xi_{5,5}\geq 38$, see Appendix
\ref{sec:system38} for the description of the Viro system that was obtained. {\tt
NewtonSLRA} is an iterative
numerical algorithm with local quadratic convergence which can solve low-rank
matrix completion problems. However, this software is not designed to handle positivity constraints so we
randomized the starting point of the iteration and ran it until it converged to
a nonnegative matrix. More specific computational methods would be needed to solve larger problems.

\providecommand{\bysame}{\leavevmode\hbox to3em{\hrulefill}\thinspace}
\providecommand{\MR}{\relax\ifhmode\unskip\space\fi MR }
\providecommand{\MRhref}[2]{%
  \href{http://www.ams.org/mathscinet-getitem?mr=#1}{#2}
}
\providecommand{\href}[2]{#2}

\appendix

\section{Polynomial system with $5$ variables, $5$ equations, $11$ monomials
and $38$ positive solutions}\label{sec:system38}
Let $C$ be the coefficient matrix
 $$C = 
\left[\begin{array}{@{~}c@{~}c@{~}c@{~}c@{~}c@{~}c@{~}c@{~}c@{~}c@{~}c@{~}c@{~}}
    \frac{14036}{26031}&\frac{-29047}{45845}&\frac{22485}{134218}&\frac{-20647}{80496}&\frac{14312}{69515}&\frac{-39015}{127243}&\frac{-6739}{42098}&\frac{19359}{360623}&\frac{16000}{83529}&\frac{1804}{131469}&\frac{4862}{44061}\\[0.5em]
    \frac{19937}{61149}&\frac{-8379}{77942}&\frac{-2105}{18949}&\frac{5635}{122379}&\frac{9229}{59989}&\frac{5391}{113671}&\frac{17593}{33547}&\frac{-50525}{112808}&\frac{-13843}{33458}&\frac{18357}{116882}&\frac{-54686}{132521}\\[0.5em]
    \frac{6391}{94296}&\frac{-3329}{144100}&\frac{7957}{156078}&\frac{-5685}{48451}&\frac{-14459}{74653}&\frac{30218}{245615}&\frac{-12227}{25927}&\frac{49127}{145204}&\frac{-14117}{47609}&\frac{29515}{59658}&\frac{-42328}{83609}\\[0.5em]
    \frac{-12249}{145219}&\frac{-13663}{97873}&\frac{-25831}{90582}&\frac{26287}{33739}&\frac{6818}{23407}&\frac{-14579}{44765}&\frac{-11126}{58889}&\frac{2247}{122770}&\frac{11139}{100537}&\frac{14421}{74818}&\frac{-60016}{644607}\\[0.5em]
  \frac{15984}{47945}&\frac{-22523}{72834}&\frac{-10734}{41165}&\frac{8531}{24837}&\frac{-21257}{47591}&\frac{22017}{37075}&\frac{5346}{284353}&\frac{19757}{194173}&\frac{5740}{83029}&\frac{-62271}{466111}&\frac{5591}{37902}
\end{array}\right].
$$

Then for $t>0$ sufficiently small, the system
$$C\cdot \begin{bmatrix}
1\\
t X_1 X_2 X_3 X_4 X_5\\
t^{2^6} X_1^2 X_2^{2^2} X_3^{2^3} X_4^{2^4} X_5^{2^5}\\
t^{3^6} X_1^3 X_2^{3^2} X_3^{3^3} X_4^{3^4} X_5^{3^5}\\
\vdots\\
t^{10^6} X_1^{10} X_2^{10^2} X_3^{10^3} X_4^{10^4} X_5^{10^5}
\end{bmatrix} = \begin{bmatrix}
  0\\0\\0\\0\\0\end{bmatrix}$$
of $5$ polynomial equations in
$\R[X_1,\ldots, X_5]$
has at least $38$ positive solutions.

\bigskip
\bigskip

\footnotesize
\noindent {\bf Authors' addresses:}

\noindent Fr\'ed\'eric Bihan, Laboratoire de Math\'ematiques, Universit\'e Savoie Mont Blanc, Campus Scientifique,
73376 Le Bourget-du-Lac Cedex, France, {\tt frederic.bihan@univ-savoie.fr}

\smallskip

\noindent Pierre-Jean Spaenlehauer, CARAMEL project, INRIA Grand-Est; Universit\'e de Lorraine; CNRS, UMR 7503; 
LORIA, Nancy, France, {\tt pierre-jean.spaenlehauer@inria.fr}
\end{document}